\documentclass[12pt, reqno]{amsart}
\usepackage{graphicx}
\usepackage{hyperref}
\usepackage[caption = false]{subfig}
\usepackage{amsthm}
\usepackage{amssymb}
\usepackage{mathrsfs}
\usepackage{mathtools}
\usepackage{enumerate}
\usepackage{amssymb}
\usepackage{wrapfig}
\usepackage{float}
\allowdisplaybreaks

\usepackage[twoside,paperwidth=190mm, paperheight=297mm, top=35mm, bottom=35mm, left=27mm, right=26mm, marginparsep=3mm, marginparwidth=40mm]{geometry}

\numberwithin{equation}{section}

\theoremstyle{plain}

\newtheorem{theorem}{Theorem}[section]
\newtheorem{corollary}[theorem]{Corollary}
\newtheorem{example}[theorem]{Example}
\newtheorem{lemma}{Lemma}[section]

\theoremstyle{definition}
\newtheorem{definition}[theorem]{Definition}
\theoremstyle{remark}
\newtheorem{remark}{Remark}[section]

\makeatother

\begin{document}

\title[Bohr-Phenomenon]{Bohr Phenomenon for $K$-Quasiconformal harmonic mappings and
	Logarithmic Power Series }
	\thanks{K. Gangania thanks to University Grant Commission, New-Delhi, India for providing Junior Research Fellowship under UGC-Ref. No.:1051/(CSIR-UGC NET JUNE 2017).}

	\author[Kamaljeet]{Kamaljeet Gangania}
	\address{Department of Applied Mathematics, Delhi Technological University,
		Delhi--110042, India}
	\email{gangania.m1991@gmail.com}

\maketitle

\begin{abstract} 
	  In this article, we establish the Bohr inequalities for the sense-preserving $K$-quasiconformal harmonic
	  mappings defined in the unit disk $\mathbb{D}$ involving classes of Ma-Minda starlike and convex univalent functions, usually denoted by $\mathcal{S}^*(\psi)$ and $\mathcal{C}(\psi)$ respectively, and for $\log (f(z)/z)$ where $f$ belongs to the Ma-Minda classes or satisfies certain differential subordination. We also estimate Logarithmic coefficient's bounds for the functions in $\mathcal{C}(\psi)$  for the case $\psi(\mathbb{D})$ be convex.  
	  
\end{abstract}
\vspace{0.5cm}

\maketitle
	
\section{Introduction}
Let $\mathcal{H}$ be the class of complex valued harmonic functions $h$ (which satisfy the Laplacian equation $\Delta{h}=4h_{z\bar{z}}=0$) defined on the unit disk $\mathbb{D}:=\{z\in \mathbb{C}: |z|<1\}$, then we can write $h=f+\bar{g}$, where $f$ and $g$ are analytic and satisfies $h(0)=g(0)$.  We say that $h$ is sense-preserving in $\mathbb{D}$ whenever the Jacobian $J_h:=|f'|^2-|g'|^2>0$. A sense-preserving homeomorphism defined in $\mathbb{D}$ which is also harmonic is called $K$-$quasiconformal$, $K\in[1,\infty)$ if the (second complex) dilatation $w_h:=g'/f'$ satisfies $|w_h(z)|\leq k$, $k=(K-1)/(K+1)\in[0,1)$.\\

In $1914$, H. Bohr~\cite{bohr1914} proved a power series inequality which is also known as the {\it classical Bohr inequality}:
\begin{theorem}[Bohr's Theorem, \cite{bohr1914}]\label{BohrTheorem}
	Let $f(z)=\sum_{m=0}^{\infty}a_mz^m$ be an analytic function in $\mathbb{D}$ and $|g(z)|<1$ for all $z\in \mathbb{D}$, then
	\begin{equation}\label{intro3-classy}
	\sum_{m=0}^{\infty}|a_m||z|^m\leq1, \quad \text{for} \quad |z|\leq\frac{1}{3}.
	\end{equation}
\end{theorem}
Since then several inequalities of similar kind are being studied in various context, which are usually named as {\it Bohr inequalities or Bohr type inequalities}. The readers are referred to see \cite{ali2017,jain2019,boas1997,bohr1914,ponnusamy17,ponnusamy18,LiuPonu-2020,muhanna10,muhanna14,muhana2016,kamal-mediter2021} and the references therein. 

For instance if in the classical Bohr inequality in Theorem~\ref{BohrTheorem}, we try to replace the initial coefficients $a_m, (m=0,1)$ by $|f(z)|$ and $|f'(z)|$, and further $z$ by some suitable choice of  functions $\omega(z)$ such that $|\omega(z)|<1$. Or replace the Taylor coefficients $a_m$ completely by the higher order derivatives of $f$. Then the combinations obtained lead us to so called Bohr-type inequalities. We now mention a few such combinations: Suppose that $f(z)=\sum_{m=0}^{\infty}a_m z^m$ be analytic in $\mathbb{D}$ and $a=|a_0|$ and ${||f_0||}_{r}^{2}=  \sum_{m=1}^{\infty}|a_{2m}|r^{2m}$, where $f_0(z)=f(z)-a_0$.
\begin{enumerate}[$(i)$]
	\item $|f(z)|^n+ \sum_{m=1}^{\infty}|a_m|r^m $, $n=0$ or $1$
	\item $|f(z)| +|f'(z)||z|+ \sum_{m=2}^{\infty}|a_m|r^m $
	\item $|f(z)|+ \sum_{m=N}^{\infty}\left| \frac{f^{(m)}(z)}{m!} \right|r^m$
	\item $|f(\omega(z))|+ \sum_{m=1}^{\infty}|a_m|r^m+ \frac{1+ar}{(1+a)(1-r)}{||f_0||}_{r}^{2}$
\end{enumerate}
For some important work in this direction, we refer to see \cite{Huang-2020,Liu-2018}.

Recall that $f$ is subordinate to  $\phi$, written as $f\prec \phi$, if $f(z)=\phi(w(z))$, where $w$ is a Schwarz function.  Further if $\phi$ is univalent, then $f\prec \phi$ if and only if $f(\mathbb{D})\subseteq \phi(\mathbb{D})$ and $f(0)=\phi(0)$. Note that the concept of subordination for analytic functions can be adopted for harmonic functions without any change, see \cite{schaubroeck-2000}. Now let us consider the class $\mathcal{A}$ which consists of analytic function with power series of the form $f(z)=z+\sum_{m=2}^{\infty}a_mz^m$, and it's important subclass of univalent functions denoted by $\mathcal{S}$. Further, the classes of Ma-Minda starlike and convex function \cite{minda94}, respectively are defined as:
\begin{equation*}
\mathcal{S}^*(\psi):= \biggl\{f\in \mathcal{A} : \frac{zf'(z)}{f(z)} \prec \psi(z) \biggl\}
\end{equation*}
and
\begin{equation*}\label{mindaclass}
\mathcal{C}(\psi):= \biggl\{f\in \mathcal{A} : 1+\frac{zf''(z)}{f'(z)} \prec \psi(z) \biggl\},
\end{equation*}
where  $\psi$ is analytic and univalent with $\Re{\psi(z)}>0$, $\psi'(0)>0$, $\psi(0)=1$ and $\psi(\mathbb{D})$ is symmetric about real axis. Note that $\psi \in \mathcal{P}$, the class of normalized Carath\'{e}odory functions. Also when $\psi(z)=(1+z)/(1-z)$, $\mathcal{S}^*(\psi)$  and $\mathcal{C}(\psi)$ reduces to the standard classes $\mathcal{S}^*$ and $\mathcal{C}$ of univalent starlike and convex functions.

In view of Muhanna~\cite{muhanna10}, the class $S(f)$ of functions $g$ subordinate to $f$ has Bohr phenomenon if for any $g(z)=\sum_{m=0}^{\infty}b_m z^m \in S(f), $ there exist an $r_0\in(0,1]$ such that
\begin{equation}\label{intro3-Muhanna}
\sum_{m=1}^{\infty}|b_m|r^m \leq d(f(0), \partial{f(\mathbb{D})})
\end{equation}
holds for $|z|=r\leq r_0,$ where $d(f(0), \partial{f(\mathbb{D})})$ denotes the Euclidean distance between $f(0)$ and the boundary of domain $f(\mathbb{D})$. When $h$ in $S(h)$ is an harmonic mapping, several Bohr inequalities have been investigated with certain assumptions on the analytic part $f$. Bhowmik and Das \cite{Bhowmik-2019} assumed that $h$ is a sense-preserving $K$-quasiconformal harmonic mapping, where the analytic part $f$ is univalent of convex univalent function, and their result states as follows:
\begin{theorem}\cite[Theorem~1]{Bhowmik-2019}\label{Bhowmik-Thm1}
	suppose that $h(z)=f(z)+ \overline{g(z)}=\sum_{m=0}^{\infty}a_m z^m+ \overline{\sum_{m=1}^{\infty}b_m {z}^m} $ be a sense-preserving $K$-quasiconformal harmonic mapping defined in $\mathbb{D}$ such that $f$ is {\it univalent} and $h_{1}(z)=f_{1}(z)+ \overline{g_{1}(z)}=\sum_{m=0}^{\infty}c_m z^m+ \overline{\sum_{m=1}^{\infty}d_m {z}^m} \in S(h)$. Then
	\begin{equation*}
	\sum_{m=1}^{\infty}(|c_m|+|d_m|) r^m \leq d(f(0), \partial{f(\mathbb{D})})
	\end{equation*}
	holds for $|z|=r\leq (5K+1-\sqrt{8K(3K+1)})/(K+1)$.
	This result is sharp for the function $p(z)= z/(1-z)^2+k\overline{z/(1-z)^2}$, where $k=(K-1)/(K+1)$.
	Moreover, if we take $f$ to be {\it convex univalent} then the result holds for
	$r \leq r_0= (K + 1)/(5K + 1)$ with sharpness for the function $q(z)= z/(1 + z) + k\overline{z/(1 - z)}$.
\end{theorem} 
Previously, several improved classical Bohr type inequality were discussed in \cite{ponnu} for the sense-preserving $K$-quasiconformal harmonic mappings such that $f$ satisfy $|f(z)|<1$ and it's applications to the corresponding analytic functions by taking $k=0$ were shown. We state here one among them:
\begin{theorem}\cite[Theorem~2.9]{ponnu}
	Let $h(z)=f(z)+ \overline{g(z)}=\sum_{m=0}^{\infty}a_m z^m+ \overline{\sum_{m=1}^{\infty}b_m {z}^m} $ be a sense-preserving $K$-quasiconformal harmonic mapping defined in $\mathbb{D}$ such that $|f(z)|<1$ and $0\leq a=|a_0|<1$. Then the following inequality holds
	\begin{equation*}
	|f(z)|+ 	\sum_{m=1}^{\infty}(|a_m|+|b_m|) r^m \leq 1,
	\end{equation*}  
	for all $a\geq \alpha_k$ and $|z|=r\leq r_{a,k}$ (the radius is sharp), where
	$$\alpha_k =\frac{\sqrt{k^2+12k 12}-(2k+3)}{k+1} \quad \text{and} \quad r_{a,k}= \frac{B_{a,k}- (k+2)(1+a)}{2a^2 (k+1)+ 2ak},$$
	where
	$B_{a,k}=\sqrt{a^2(k^2+8k+8)+ 2a(k^2+6k+4)+ (k+2)^2}.$
\end{theorem}

Motivated by the above results and noticing the role of sharp coefficient bounds for functions in a given class to establish such inequalities, see~\cite{jain2019,kamal-mediter2021,hamada-2021,ganga-iranian,gangania-bohr}, and observing that sharp coefficients bounds are not available in general for the class $\mathcal{S}^*(\psi)$ and $\mathcal{C}(\psi)$. It is promising to discuss the case where analytic part $f$ of a sense-preserving $K$-quasiconformal harmonic mapping belongs to the general class $\mathcal{S}^*(\psi)$ and $\mathcal{C}(\psi)$. More precisely,
\begin{definition}
	Let $f\in \mathcal{S}^*(\psi)$ (or $\mathcal{C}(\psi)$), and $h(z)=f(z)+ \overline{g(z)}=z+\sum_{m=2}^{\infty}a_m z^m+ \overline{\sum_{m=1}^{\infty}b_m {z}^m} $ be a sense-preserving $K$-quasiconformal harmonic mapping defined in $\mathbb{D}$. Then
	\begin{equation*}
	S_{\psi}(h):= \left\{h_{1}(z) : h_{1}(z) \prec h(z)  \right\},
	\end{equation*}
	where $h_{1}(z)=f_{1}(z)+ \overline{g_{1}(z)}=z+\sum_{m=2}^{\infty}c_m z^m+ \overline{\sum_{m=1}^{\infty}d_m {z}^m}$.
\end{definition}

Similarly, in the literature of coefficient's problems, much of the attention has been on the inequalities related to the {\it logarithmic coefficients}, see \cite{Milin-history1985}, due to it's important role in settling the Bieberbach conjecture, we refer to see the recent articles \cite{choLog-2019,PonuSugawa-Log2021,Roth-2007} and their references, which are defined by the following logarithmic power series
\begin{equation}\label{3-logseries}
\log\left(\frac{f(z)}{z} \right)= 2\sum_{m=1}^{\infty} \gamma_m z^m,
\end{equation}
where $f\in \mathcal{S}$. Recently, Adegani et. al~\cite[Theorem~1, Sec~2]{choLog-2019} settled the problem of sharp bounds for $\gamma_m$ for the functions in the class $\mathcal{S}^*(\psi)$, but obtained only three initial logarithmic coefficient's bounds for the class  $\mathcal{C}(\psi)$ in \cite[Theorem~2, Sec~2]{choLog-2019}. Therefore, we consider the problem of Bohr inequality similar to \eqref{intro3-classy} for $\log(f(z)/z)$, where $f$ either belongs to the classes $\mathcal{S}^*(\psi)$ and $\mathcal{C}(\psi)$ or satisfies certain standard differential subordination\cite{subbook}. The reason why we consider the Bohr inequality of classical type for the series \eqref{3-logseries} instead of \eqref{intro3-Muhanna} follows with the observations discussed by Bhowmik and Das: The quantity $d(g, \partial{\Omega})$, where $\Omega$ is the image of $\mathbb{D}$ under $g(z):=\log(f(z)/z)$, can be arbitrarily small positive number for $f\in \mathcal{S}$ which is observed with the help of univalent polynomials $f_n(z)=z+(z^2/n)$ for each $n\geq2$ such that the image of $\log(f_n(z)/z)$ does not include the point $\log(1+(1/n))$. 

While proving our result related to Bohr inequality for the logarithmic power series, we surprisingly get estimates for the logarithmic coefficient's bounds for the class $\mathcal{C}(\psi)$ for the case $\psi(\mathbb{D})$ being convex, see Theorem~\ref{logcoef-convexclass}. {\it Note that till date obtaining sharp bounds even for a particular class of convex functions is an open problem}.

\section{Bohr phenomenon for $K$-quasiconformal mappings}

For convenience, let us consider the functions $f_n \in \mathcal{S}^*(\psi)$, $n\in \{0,1,2,\cdots\}$ defined in the unit disk $\mathbb{D}$ as 
\begin{equation}\label{extremals3}
\frac{zf'_n(z)}{f_n(z)}= \psi(z^{n+1}).
\end{equation}
In case when the coefficients of $f_n$ in its power series expansion are positive, we denote $f_n$ by $\hat{f_n}$. Since sharp bounds for the Taylor coefficients in general for the functions in the class $\mathcal{S}^*(\psi)$ are yet not known. Therefore, we need the following result, obtained in \cite[Lemma~2.1, Sec~2]{kamal-mediter2021}.

\begin{lemma}\label{series-lem}
	let $f(z)=\sum_{n=0}^{\infty}a_n z^n$ and $g(z)=\sum_{k=0}^{\infty}b_k z^k$ be analytic in $\mathbb{D}$ and $g\prec f$, then 
	\begin{equation*}
	\sum_{k=N}^{\infty}|b_k|r^k \leq \sum_{n=N}^{\infty}|a_n|r^n
	\end{equation*}
	for $|z|=r\leq \frac{1}{3}$ and $N\in \mathbb{N}$.
\end{lemma}
which was further used to prove \cite[Corollary~2.1, Sec~2]{kamal-mediter2021}, and there replacing $g(z)$ by $g(z)/M$ for $M>0$ we easily get:

\begin{lemma}\label{quasi-serieslem}
	Let the analytic functions $f, g$ and $h$ satisfies $g(z)=M\phi(z)f(\omega(z))$ in $\mathbb{D}$, where $\omega$ is the Schwarz function. Assume $|\phi(z)|\leq \tau$ for $|z|< \tau\leq1$. Then
	\begin{equation*}
	\sum_{k=N}^{\infty}|b_k|r^k \leq \tau M \sum_{n=N}^{\infty}|a_n|r^n, \quad\text{for}\quad 0\leq |z|=r\leq \frac{\tau}{3},
	\end{equation*}
	where $M>0$ and $N\in\mathbb{N}$.
\end{lemma}
The Lemma~\ref{quasi-serieslem}  generalizes the result of Alkhaleefah et al.~\cite[Theorem~2.1]{ponnu}. For other inequalities of this kind, see \cite{MuhannaOperator-2021}.

Now the following result is a generalization of Theorem~\ref{Bhowmik-Thm1} in the sense that $f$ be a Ma-Minda univalent starlike functions, while it's convex analogue is presented in Theorem~\ref{K-quasi3-conv}.	
\begin{theorem}\label{K-quasi3}
	Let $h(z)=f(z)+ \overline{g(z)}=z+\sum_{m=2}^{\infty}a_m z^m+ \overline{\sum_{m=1}^{\infty}b_m {z}^m} $ be a sense-preserving $K$-quasiconformal harmonic mapping defined in $\mathbb{D}$ such that $f\in \mathcal{S}^*(\psi)$ and $h_{1}(z)=f_{1}(z)+ \overline{g_{1}(z)}=z+\sum_{m=2}^{\infty}c_m z^m+ \overline{\sum_{m=1}^{\infty}d_m {z}^m} \in S_{\psi}(h)$. Then
	\begin{equation}\label{bohrcombi3.1}
	\sum_{m=1}^{\infty}(|c_m|+|d_m|) r^m \leq d(f(0), \partial{f(\mathbb{D})})
	\end{equation}
	holds for $|z|\leq r^*=\min\{r_0, \frac{1}{3} \}$, where $r_0$ is the unique root in $(0,1)$ of 
	\begin{equation*}
	\frac{2K}{K+1}\hat{f}_0(r)=-f_0(-1).
	\end{equation*}
	The result is sharp for the function $p(z)=f_0(z)+k \overline{f_0(z)}$, where $\hat{f_0}=f_0$ as defined in \eqref{extremals3} when $r^*=r_0$.
\end{theorem}
\begin{proof}
	Clearly, $f'(z)\neq0$ as $f\in \mathcal{S}^*(\psi)$. From the dilatation $w_h(z)= g'/f'$, we have $|w_f|\leq k=(K-1)/(K+1) <1$. We consider the case $w_h$ being non constant, and the case $w_f(z)=c k$, $|c|=1$ can be handled on similar lines. Thus, by Maximum-Modulus principle there exist $\phi: \mathbb{D} \rightarrow \mathbb{D}$ such that
	\begin{equation}\label{dilatation3}
	g'(z)= k\phi(z) f'(z).
	\end{equation}
	Now in Lemma~\ref{quasi-serieslem} setting $\tau=1=N$ and $w(z)=z$ with $p(z)=\sum_{n=0}^{\infty}p_nz^n$ and $q(z)=\sum_{n=0}^{\infty} q_nz^n$ such that $p(z) = M \phi(z)q(z)$ for some $M > 0$, $\phi: \mathbb{D} \rightarrow \mathbb{D}$, we have for $|z|=r\leq 1/3$
	\begin{equation*}
	\sum_{n=0}^{\infty}|q_n| r^n \leq M \sum_{n=0}^{\infty}|p_n|r^n,
	\end{equation*}	
	which on applying to \eqref{dilatation3} and Lemma~\ref{series-lem} with $N=1$ on $f(z)/z \prec f_0(z)/z$, and then integrating from $0$ to $r$ finally yield
	\begin{equation}\label{coeffcomparision3}
	\sum_{m=1}^{\infty}|b_m|r^m \leq k \sum_{m=1}^{\infty} |a_m| r^m \leq k \hat{f_0}(r)
	\end{equation}
	for $|z|=r\leq 1/3$. Thus from \eqref{coeffcomparision3}, we have for $|z|=r\leq 1/3$
	\begin{equation}\label{finalcoefcombination3}
	\sum_{m=1}^{\infty} |a_m|r^m + \sum_{m=1}^{\infty} |b_m| r^m  \leq (1+k) \sum_{m=1}^{\infty}|a_m|r^m. 
	\end{equation}
	Now let us consider the function 
	\begin{equation*}
	T(r):= (1+k)\hat{f_0}(r)+ f_0(-1), \quad 0\leq r\leq1.
	\end{equation*}
	Then it is a continuous function of $r$ and $T'(r)\geq 0$. Also $T(0)=f_0(-1)<0$ and $T(1)>0$. Therefore, $T(r)$ has a unique root, say $r_0$ in $(0,1)$. Hence by \eqref{finalcoefcombination3}
	\begin{equation*}
	\sum_{m=1}^{\infty}|a_m| r^m+ \sum_{m=1}^{\infty} |b_m| r^m \leq d(f(0), \partial{f(\mathbb{D})})
	\end{equation*}
	holds for $|z|=r\leq \min\{r_0, 1/3 \}$, where $r_0$ is the root of $T(r)$. Now it remains to prove that
	\begin{equation*}
	\sum_{m=1}^{\infty}(|c_m|+|d_m|) r^m \leq \sum_{m=1}^{\infty}|a_m| r^m+ \sum_{m=1}^{\infty} |b_m| r^m
	\end{equation*}
	is true for $|z|\leq r^*$ which in fact holds, since $f_1 \prec f$ and $g_1 \prec g$ \cite[P.~164, Sec~2]{schaubroeck-2000} with the application of Lemma~\ref{series-lem} yields $	\sum_{m=1}^{\infty}|c_m| r^m \leq \sum_{m=1}^{\infty}|a_m| r^m$ and $\sum_{m=1}^{\infty}|d_m| r^m \leq \sum_{m=1}^{\infty} |b_m| r^m$ for $r\leq1/3$. Sharpness of the result follows for the function $p(z)$ as defined in the hypothesis by direct computation for $|z|=r\leq r^*$ such that equality holds in \eqref{bohrcombi3.1}. \qed
\end{proof}	

In the next result, we establish the Bohr-Rogosinski phenomenon for the the class $S_{\psi}(h)$ following the proof of Theorem~\ref{K-quasi3}. Note that if we take $N\rightarrow 1$ and $n\rightarrow \infty$ in the result below, we trace back Theorem~\ref{K-quasi3}.
\begin{corollary}
	Let $h(z)=f(z)+ \overline{g(z)}=z+\sum_{m=2}^{\infty}a_m z^m+ \overline{\sum_{m=1}^{\infty}b_m {z}^m} $ be a sense-preserving $K$-quasiconformal harmonic mapping defined in $\mathbb{D}$ such that $f\in \mathcal{S}^*(\psi)$ and $h_{1}(z)=f_{1}(z)+ \overline{g_{1}(z)}=z+\sum_{m=2}^{\infty}c_m z^m+ \overline{\sum_{m=1}^{\infty}d_m {z}^m} \in S_{\psi}(h)$. Then 
	\begin{equation}
	|f(z^n)|+   \sum_{m=N}^{\infty}(|c_m|+|d_m|) r^m \leq d(f(0), \partial{f(\mathbb{D})})
	\end{equation}
	holds for $|z|\leq r^*=\min\{r_0, \frac{1}{3} \}$, where $n\in \mathbb{N}$ and $r_0$ is the unique root in $(0,1)$ of 
	\begin{equation*}
	(K+1)(\hat{f_0}(r^n) + f_0(-1))+ 2K(\hat{f_0}(r)-S_{N}(\hat{f_0}))=0,
	\end{equation*}
	where $S_N(f)=\sum_{m=1}^{N-1}a_mr^m$.
	The result is sharp for the function $p(z)=f_0(z)+k \overline{f_0(z)}$, where $\hat{f_0}=f_0$ as defined in \eqref{extremals3} when $r^*=r_0$.
\end{corollary}
\begin{proof}
	Since $|f(z^n)| \leq \hat{f_0}(r^n)$ for $|z|=r$. Therefore, applying Lemma~\ref{series-lem} on \eqref{dilatation3}, and  Lemma~\ref{series-lem} on $f(z)/z \prec f_0(z)/z$, and then integrating from $0$ to $r$, we deduce that
	\begin{align*}
	& \quad |f(z^n)|+   \sum_{m=N}^{\infty}(|c_m|+|d_m|) r^m \\
	&\leq |\hat{f_0}(r^n)|+   (1+k)\sum_{m=N}^{\infty} |a_m| r^m \\
	&\leq  -f_0(-1) \leq d(f(0), \partial{f(\mathbb{D})})
	\end{align*}
	holds for $|z|=r\leq \{r_0,1/3 \}$, where $r_0\in(0,1)$ is the unique positive root of the equation 
	\begin{equation*}
	T_N(r) := \hat{f_0}(r^n)+f_0(-1)+(1+k) \sum_{m=N}^{\infty}|a_m|r^m=0.
	\end{equation*}
	Existence of the root $r_0$ follows, since $T_N(r)$ is continuous function of $r$ and $T'_N(r)\geq0$ with $T_N(0)=f_0(-1)$ and $T_N(1)>0$. \qed
\end{proof}

\begin{corollary}\cite[Theorem~5.1]{gangania-bohr}\cite[Theorem~3.1]{hamada-2021}
	Let $f\in \mathcal{S}^*(\psi)$ and $f_{1}(z)=\sum_{m=1}^{\infty}c_m z^m \in S_{\psi}(f)$. Then
	\begin{equation*}
	\sum_{m=1}^{\infty}|c_m| r^m \leq d(f(0), \partial{f(\mathbb{D})})
	\end{equation*}
	holds for $|z|\leq r^*=\min\{r_0, \frac{1}{3} \}$, where $r_0$ is the unique root in $(0,1)$ of 
	\begin{equation*}
	\hat{f}_0(r)=-f_0(-1).
	\end{equation*}
	The result is sharp for the function $\hat{f_0}=f_0$ as defined in \eqref{extremals3} when $r^*=r_0$.
\end{corollary}	

\begin{corollary}\label{starlike3}
	Let $f \in \mathcal{S}^*$ in Theorem~\ref{K-quasi3}. Then the inequality \eqref{bohrcombi3.1} holds for $|z|=r\leq r_0<1/3$, where
	\begin{equation*}
	r_0= (5K+1-\sqrt{8K(3K+1)})/(K+1).
	\end{equation*} 
	The radius is sharp.
\end{corollary}

\begin{corollary}
	Let $\psi(z)=\frac{1+Dz}{1+Ez}$ in Theorem~\ref{K-quasi3}, where $-1\leq E< D\leq 1$. Then the inequality \eqref{bohrcombi3.1} holds for $|z|=r\leq \min\{r_0,1/3\}$, where $r_0$ is the unique root of the equation
	\begin{equation*}
	\frac{2K}{K+1}(r+\sum_{m=2}^{\infty}\prod_{t=0}^{m-2}\frac{|E-D+Et|}{t+1}r^m )-(1-E)^{\frac{D-E}{E}}=0.
	\end{equation*}
	Further assume that $f_0=\hat{f_0}$ as defined in \eqref{extremals3}, and
	\begin{enumerate}[$(i)$]
		\item If $E\neq0$ and $3(1-E)^{\frac{D-E}{E}} \leq (1+\frac{E}{3})^{\frac{D-E}{E}}$;
		\item If $E=0$ and $D\geq \frac{3}{4}\log{3}$.
	\end{enumerate}
	Then the radius $r_0$ is sharp.
	
\end{corollary}

\begin{corollary}\label{starlikeorderalpha3}
	Let $\psi(z)=\frac{1+(1-2\alpha)z}{1-z}$ in Theorem~\ref{K-quasi3}, where $0\leq \alpha \leq 1/2$.
	Then the inequality \eqref{bohrcombi3.1} holds for $|z|=r\leq r_0<1/3$, where $r_0$ is the unique root of the equation
	\begin{equation*}
	K 2^{2(1-\alpha)+1}r-(K+1)(1-r)^{2(1-\alpha)}=0.
	\end{equation*}
	The radius is sharp. 
\end{corollary}

\begin{remark}\label{inclusion3.1}
	Substituting $\alpha =0$ in Corollary~\ref{starlikeorderalpha3}, we retrace Corollary~\ref{starlike3}. One should also note that if we take $K=1$ and $\alpha=0$ in Corollary~\ref{starlike3}, then the radius $r_0=3-2\sqrt{2}$, which is equal to the Bohr radius for the class $\mathcal{S}^*$. Moreover, thinking in the direction of inclusion relationship  for a class $\mathcal{M}$ of starlike functions such that $\mathcal{M}\subset\mathcal{S}^*(\alpha)$ for some $\alpha$, then we observe that Corollary~\ref{starlikeorderalpha3} provides a lower bound for the radius $r_0$ in Theorem~\ref{K-quasi3} for the case of  $\mathcal{M}$ in place of $\mathcal{S}^*(\psi)$ in general.
\end{remark}	

Now recall that a function $f\in \mathcal{S}$ which has the property that for every circular arc $\Gamma$ contained in $\mathbb{D}$ with center $\xi\in \mathbb{D}$, the image arc $f(\Gamma)$ is a starlike arc with respect to $f(\xi)$, is called Goodman uniformly starlike ($\equiv \mathcal{UST}$). Similary the class of uniformly convex function ($\equiv \mathcal{UCV}$) is also defined. Motivated by this class recently Darus introduced the class:

Let $f\in \mathcal{A}$. Then $f\in k-\mathcal{UCST}(\alpha)$ if and only if
\begin{equation*}
\Re\bigg\{\frac{(zf'(z))'}{f(z)}  \bigg\} > k\bigg| (1-\alpha)\frac{zf'(z)}{f(z)}+\alpha \frac{(zf'(z))'}{f(z)}-1 \bigg|,
\end{equation*}
$k\geq0$, $0\leq \alpha\leq1$. Now following the Remark~\ref{inclusion3.1} and the inclusion relation given in \cite[Theorem~2.1]{Darus-2018}, Corollary~\ref{starlikeorderalpha3} gives:

\begin{example}
	Let us replace $\mathcal{S}^*(\psi)$ by $k-\mathcal{UCST}(\alpha)$ in Theorem~\ref{K-quasi3}.
	Then the inequality \eqref{bohrcombi3.1} holds for $|z|=r\leq r_0<1/3$, where $r_0$ is the unique root of the equation
	\begin{equation*}
	\frac{2Kr}{(1-r)^{2(1-\delta)}}-\frac{K+1}{4^{(1-\delta)}}=0,
	\end{equation*}
	where 
	\begin{equation*} 
	\delta=\frac{(2\gamma-\beta)+\sqrt{(2\gamma-\beta)^2 +8\beta}}{4}
	\end{equation*}
	and $0\leq\delta<1$, $\beta=\frac{1+\alpha k}{1+k}$ and $\gamma=\frac{1}{1+k}$ and further satisfy
	\begin{equation*}
	(1-\alpha)k^2-(1+\alpha)k-2\geq0.
	\end{equation*}
	In view of the Remark~\ref{inclusion3.1}, note that  $k-\mathcal{UCST}(\alpha) \subset \mathcal{S}^*(\delta)$, and therefore the radius $r_0$ can be further improve.
\end{example}

Now we conclude this section with the following result which is convex analogue of Theorem~\ref{K-quasi3}:

\begin{theorem}\label{K-quasi3-conv}
	Let $h(z)=f(z)+ \overline{g(z)}=z+\sum_{m=2}^{\infty}a_m z^m+ \overline{\sum_{m=1}^{\infty}b_m {z}^m} $ be a sense-preserving $K$-quasiconformal harmonic mapping defined in $\mathbb{D}$ such that $f\in \mathcal{C}(\psi)$ and $h_{1}(z)=f_{1}(z)+ \overline{g_{1}(z)}=z+\sum_{m=2}^{\infty}c_m z^m+ \overline{\sum_{m=1}^{\infty}d_m {z}^m} \in S_{\psi}(h)$. Then
	\begin{equation*}
	\sum_{m=1}^{\infty}(|c_m|+|d_m|) r^m \leq d(f(0), \partial{f(\mathbb{D})})
	\end{equation*}
	holds for $|z|\leq r^*=\min\{r_0, \frac{1}{3} \}$, where $r_0$ is the unique root in $(0,1)$ of 
	\begin{equation*}
	\frac{2K}{K+1}\hat{f}_0(r)=-f_0(-1).
	\end{equation*}
	The result is sharp for the function $p(z)=f_0(z)+k \overline{f_0(z)}$ when $r^*=r_0$, where $\hat{f_0}=f_0$ is the solution of $$1+\frac{zf''(z)}{f'(z)}=\psi(z).$$
\end{theorem}
\begin{proof}
	It follows from Theorem~\ref{K-quasi3} by suitably using Alexander's relation between starlike and convex functions. \qed
\end{proof}	

\begin{corollary}
	Let $f \in \mathcal{C}$ in Theorem~\ref{K-quasi3-conv}. Then the inequality \eqref{bohrcombi3.1} holds for $|z|=r\leq r_0<1/3$, where
	\begin{equation*}
	r_0= (K+1)/(5K+1).
	\end{equation*} 
	The radius is sharp.
\end{corollary}
The above corollary was obatined by Liu and Ponnusamy~\cite[Theorem~1]{Liu-Ponnu-2021}, (also see \cite[Theorems~1.1 and 1.3]{Kayumov-PonShakirov2018}), where the analytic part $f$ was not normalized.

\section{Bohr phenomenon for Logarithmic power series}

Let us recall that logarithmic coefficients $\gamma_m$ of the functions $f\in \mathcal{S}^*(\psi)$ are defined as
\begin{equation}\label{logcoef3}
\log \bigg(\frac{f(z)}{z} \bigg)= 2\sum_{m=1}^{\infty} \gamma_m z^m.
\end{equation}

In the following theorem, we generalize and provide a simple proof of the result \cite[Theorem~4, Sec~3]{Bhowmik-2019}. 
\begin{theorem}\label{logbohrstarlike}
	Let $\psi(z)=1+B_1 z+B_2 z^2+ \cdots$ and $f(z)=z+\sum_{m=2}^{\infty}a_m z^m \in \mathcal{S}^*(\psi) $, where $\psi(\mathbb{D})$ is convex. Then the logarithmic coefficient's of $f$ satisfies
	\begin{equation*}
	2\sum_{m=1}^{\infty} |\gamma_m| r^m \leq 1, 
	\quad \text{for} \quad |z|=r \leq  1-\frac{1}{\exp\big(\frac{1}{|B_1|}\big)}.
	\end{equation*}
	The result is sharp for the case when the function $f$ defined as $zf'(z)/f(z)=\psi(z)$ is an extremal for the logarithmic coefficient's bounds.
\end{theorem}
\begin{proof}
	Recall the result \cite[Theorem~1, Sec~2]{choLog-2019} which says: If $f\in \mathcal{S}^*(\psi)  $, then $\gamma_m$ as defined in \eqref{logcoef3} satisfies
	\begin{equation}
	|\gamma_m| \leq \frac{|B_1|}{2m},
	\end{equation}
	whenever $\psi(\mathbb{D})$ is convex. Hence 
	\begin{align*}
	2\sum_{m=1}^{\infty} |\gamma_m|r^m \leq 2\sum_{m=1}^{\infty} \frac{|B_1|}{2m}r^m
	=|B_1|\sum_{m=1}^{\infty}\frac{1}{m}r^m.
	\end{align*}
	Since the sum of the power series $\sum_{m=1}^{\infty}\frac{x^m}{m}=-\log(1-x)$ for  $x\in (-1,1]$. Therefore, we have
	\begin{equation*}
	2\sum_{m=1}^{\infty}|\gamma_m|r^m \leq |B_1|\log\bigg( \frac{1}{1-r}\bigg) \leq 1,
	\end{equation*}
	which holds whenever $r\leq 1-{1}/{(\exp\big(\frac{1}{|B_1|}\big))}$. 
	\qed
	
\end{proof}	


Now we apply the above theorem to some well-known and recently introduced classes:
\begin{corollary}\label{all-result}
	Let $g\in \mathcal{A}$ and $f\in \mathcal{S}^*(\psi)$, where $\psi(\mathbb{D})$ is convex. Then
	\begin{equation*}
	2\sum_{m=1}^{\infty} |\gamma_m| r^m \leq 1, 
	\quad \text{for} \quad |z|=r \leq r_0=  1-\frac{1}{\exp\big(\frac{1}{|B_1|}\big)},
	\end{equation*}
	follows for each one of the following cases:
	\begin{itemize}
		\item [$(i)$]  	$B_1=D-E$ when $\psi(z)= \frac{1+Dz}{1+Ez}$, where $-1\leq E<D\leq1$.
		
		\item [$(ii)$]	$B_1=2(1-\alpha)$ when $\psi(z)=\frac{1+(1-2\alpha)z}{1-z}$, where $0\leq\alpha<1$.
		
		\item [$(iii)$] $B_1=2\eta$ when $\psi(z)=\left(\frac{1+z}{1-z}\right)^{\eta}$, where $0<\eta \leq1$.
		
		\item [$(iv)$] 	$B_1= (5\sqrt{2}-6)/(2\sqrt{2})$ when  $\psi(z)=\sqrt{2}-(\sqrt{2}-1)\sqrt{\frac{1-z}{1+2(\sqrt{2}-1)z}}$.
		
		\item [$(v)$] 	$B_1={b^{\frac{1}{a}}}/{a}$ when $\psi(z)=(b(1+z))^{1/a}$, where $a\geq1$ and $b\geq 1/2$.

		\item [$(vi)$] $B_1={1-\alpha}$ when $\psi(z)=\alpha+(1-\alpha)e^z$, where $0\leq\alpha<1$.

		\item [$(vii)$] $B_1={(1-\alpha)}/{2}$ when $\alpha+(1-\alpha) \sqrt{1+z}$, where $0\leq\alpha<1$.

		\item [$(vii)$] $B_1= 1/2$ when $\psi(z)=\frac{2}{1+e^{-z}}$.
		
		
	\end{itemize}
	The result is sharp for the case when the function $f$ defined as $zf'(z)/f(z)=\psi(z)$ is an extremal for the logarithmic coefficient's bounds.
\end{corollary}

\begin{remark}
	In Theorem~\ref{logbohrstarlike}, let $f(z)=z+\sum_{m=2}^{\infty}a_m z^m \in \mathcal{S}^*(\psi) $, where $\psi(\mathbb{D})$ is starlike with respect to $1$ but convex. Then the logarithmic coefficients of $f$ satisfies, see \cite[Theorem~1, Sec~2]{choLog-2019},
	\begin{equation*}
	|\gamma_m| \leq \frac{|B_1|}{2}, \quad n\in \mathbb{N},
	\end{equation*}
	and following the Proof of Theorem~\ref{logbohrstarlike},  we deduce that
	\begin{equation*}
	2\sum_{m=1}^{\infty} |\gamma_m| r^m \leq 1 
	\quad \text{for} \quad |z|=r \leq (1+|B_1|)^{-1}.
	\end{equation*}
\end{remark}	

Note that all sharp logarithmic coefficient's bounds for the functions in the class $\mathcal{C}(\psi)$ of normalized convex function are not available till now, see \cite[Theorem~1, Sec~2]{choLog-2019}. But using the technique of differential subordination, we overcome this difficulty to achieve the Bohr radius for logarithmic power series for the Ma-Minda classes of convex functions in the following result.
\begin{theorem}\label{logbohrconvex}
	Let $\phi$ be convex in $\mathbb{D}$ and
	suppose $\psi$ be the convex solution of the differential equation
	\begin{equation}\label{briot-sol}
	\psi(z)+\frac{z\psi'(z)}{\psi(z)}=\phi(z).
	\end{equation}
	If $f\in \mathcal{C}(\phi)$ with $\phi(z)=1+B_1z+B_2z^2+\cdots$. Then the logarithmic coefficients of $f$ satisfies
	\begin{equation}\label{logBohrineq}
	2\sum_{m=1}^{\infty} |\gamma_m| r^m \leq 1, 
	\end{equation}
	for
	\begin{equation*}
	|z|=r \leq  1-\frac{1}{\exp\big(\frac{2}{|B_1|}\big)},
	\end{equation*}
	The result is sharp for the case when the function $f$ defined as $1+zf''(z)/f'(z)=\phi(z)$ is an extremal for the logarithmic coefficient's bounds.
	
\end{theorem}
\begin{proof}
	Let us define $p(z):=zf'(z)/f(z)$. Since  $f\in \mathcal{C}(\phi)$, therefore we have $1+zf''(z)/f'(z)\prec \phi(z)$, which can be equivalently written as
	\begin{equation}\label{briot}
	p(z)+\frac{zp'(z)}{p(z)}=1+\frac{zf''(z)}{f'(z)} \prec \phi(z).
	\end{equation}
	Since $\Re\phi(z)>0$ and $\phi$ is convex in $\mathbb{D}$, therefore using \cite[Theorem~3.2d, p.~86]{subbook} the solution $\psi$ of the differential equation \eqref{briot-sol} 
	is analytic in $\mathbb{D}$ with $\Re\psi(z)>0$ and has the following integral form given by
	\begin{equation}\label{Taylor-logbohrconvex}
	\psi(z):= h(z)\left(\int_{0}^{z}\frac{h(t)}{t}dt\right)^{-1}=1+\frac{B_1}{2}z+\frac{{B_1}^2+4B_2}{12} z^2+\cdots,
	\end{equation}
	where 
	\begin{equation*}
	h(z)= z\exp\int_{0}^{z}\frac{\phi(t)-1}{t}dt.
	\end{equation*}
	Since $\Re\psi(z)>0$ and $p$ satisfies the subordination~\eqref{briot}, therefore using \cite[Lemma~3.2e, p.~89]{subbook} we conclude that $\psi$ is univalent and $p\prec \psi$, where $\psi$ is the best dominant. 
	Thus we have obtained that $f\in\mathcal{C}(\phi)$ implies $zf'(z)/f(z) \prec \psi(z)$ and $\psi$ is the best dominant, which is a univalent Carath\'{e}odory function. Now applying Theorem~\ref{logbohrstarlike}, the desired radius is achieved.	
	\qed
\end{proof}

\begin{remark}
	If $f\in \mathcal{C}$. Then the logarithmic coefficients of $f$ satisfies
	\begin{equation*}
	2\sum_{m=1}^{\infty} |\gamma_m| r^m \leq 1, 
	\end{equation*}
	for
	\begin{equation*}
	|z|=r \leq  1-1/e.
	\end{equation*}
	and the radius is sharp for the function $f(z)=z/(1-z)$. The result was obtained by Bhowmik and Das in \cite[Remark on Page no~741]{Bhowmik-2019}.
\end{remark}

Now we estimate the Logarithmic coefficient's bounds for the functions in the class $\mathcal{C}(\psi)$ for the case $\psi(\mathbb{D})$ being convex, which generalize the result \cite[Theorem~2, Sec~2]{choLog-2019} where sharp two initial Logarithmic coefficient's bounds and estimate for the third  bound ($\gamma_m$, for $m=1,2$ and $3$) were discussed.
\begin{theorem}{ (Convex analogue of \cite[Theorem~1, Sec~2]{choLog-2019})}\label{logcoef-convexclass}
	Let $\phi$ be convex in $\mathbb{D}$ 
	and
	suppose $\psi$ be the solution of the differential equation
	\begin{equation*}
	\psi(z)+\frac{z\psi'(z)}{\psi(z)}=\phi(z).
	\end{equation*}
	If $f\in \mathcal{C}(\phi)$ with $\phi(z)=1+B_1z+B_2z^2+\cdots$. Then the logarithmic coefficients of $f$ satisfies the inequalities:
	\begin{enumerate}
		\item If $\psi$ is convex. Then
		\begin{equation}\label{5}
		|\gamma_m| \leq \frac{|B_1|}{4m}, \quad m\in \mathbb{N},
		\end{equation}
		\begin{equation*}
		\sum_{m=1}^{n}|\gamma_m|^2 \leq \frac{1}{4}\sum_{m=1}^{n}\frac{|c_m|^2}{m^2}, \quad n\in \mathbb{N},
		\end{equation*}
		and
		\begin{equation}\label{7}
		\sum_{m=1}^{\infty}|\gamma_m|^2 \leq \frac{1}{4}\sum_{m=1}^{\infty}\frac{|c_m|^2}{m^2}
		\end{equation} 
		where $c_m$ are the Taylor series coefficients of $\psi$ (see Eq.~\eqref{Taylor-logbohrconvex}).
		
		\item If $\psi$ is starlike with respect to $1$. Then
		\begin{equation}\label{8}
		|\gamma_m| \leq \frac{|B_1|}{4}, \quad m\in\mathbb{D}.
		\end{equation} 
		The inequalities $\eqref{5}$, $\eqref{7}$ and $\eqref{8}$ are sharp for the case when the function $f$ defined as $zf'(z)/f(z)=\psi(z)$ serves as an extremal for the logarithmic coefficient's bounds.    
	\end{enumerate}
	
\end{theorem}
\begin{proof}
	Following the Proof of Theorem~\ref{logbohrconvex}, and then applying the result \cite[Theorem~1, Sec~2]{choLog-2019} completes the proof. \qed
\end{proof}

In our next result, we show the application to the Janowski class, which covers many well-known classes. Here $\mathcal{C}[D,E]:=\mathcal{C}(({1+Dz})/({1+Ez})$.

Let us first consider the confluent and Gaussian hypergeometric functions, respectively as follows:
\begin{equation}\label{janw-solution}
q(z)=
\left\{
\begin{array}{lr}
{}_{2}F_{1}\left(1-\frac{D}{E}, 1, 2; \frac{Ez}{1+Ez} \right), & \text{if}\; E\neq0;\\

{}_{1}F_{1}\left(1,2;-Dz\right), & \text{if}\; E=0.
\end{array}
\right.
\end{equation} 
We note that for $1+D/E\geq0$ and $-1\leq E<0$, 
\begin{equation*}
\min_{|z|=r}\Re\psi(z) = {\psi(-r)}=\frac{1}{q(-r)} >0, \; \text{where}\; \psi(z)=1/q(z).
\end{equation*}

\begin{corollary}\label{AB}
	Let $f$ belongs to $\mathcal{C}[D,E]$, where $-1\leq E<D\leq1$, and in addition $q$ as defined in \eqref{janw-solution} be convex.  Then the logarithmic coefficients of $f$ satisfies \eqref{logBohrineq} for
	\begin{equation*}
	|z|=r \leq
	\left\{
	\begin{array}{lr}
	1-\frac{1}{\exp\big(\frac{2}{D-E}\big)}, & \text{if}\; E, D\neq0;\\\\
	
	1-\frac{1}{\exp\big(\frac{2}{|E|}\big)}, & \text{if}\; E\neq0, D=0,
	
	\end{array}
	\right.
	\end{equation*}
\end{corollary}
\begin{proof}
	In Theorem~\ref{logbohrconvex}, put $\phi(z)= (1+D z)/(1+E z)$. Then we have $\psi(z):= 1/q(z)$, where
	\begin{equation*}
	q(z)=
	\left\{
	\begin{array}{lr}
	\int_{0}^{1}\left(\frac{1+Etz}{1+Ez}\right)^{\frac{D-E}{E}}dt, & \text{if}\; E\neq0;\\\\

	\int_{0}^{1}e^{D(t-1)z}dt, & \text{if}\; E=0,
	\end{array}
	\right.
	\end{equation*} 
	which with some computation can be written explicitly as
	\begin{equation*}
	\psi(z)=
	\left\{
	\begin{array}{lr}
	\frac{Dz (1+Ez)^{\frac{D}{E}-1}}{-1+(1+Ez)^{\frac{D}{E}} }=1+\frac{D-E}{2}z+ \frac{D^2-6DE+5E^2}{12}z^2+\cdots, & \text{if}\; E, D\neq0;\\\\
	
	\frac{Ez}{(1+Ez)\log(1+Ez)}=1-\frac{E}{2}z+ \frac{5E^2}{12}z^2-\cdots, & \text{if}\; E\neq0, D=0;\\\\
	
	\frac{Dz e^{Dz}}{e^{Dz}-1}, & \text{if}\; E=0,
	\end{array}
	\right.
	\end{equation*}
	and satisfies the differential Eq.~\eqref{briot-sol}. Hence, the result follows from Theorem~\ref{logbohrstarlike}. \qed
\end{proof}

Now we have the result for the class of convex functions of order $\alpha$ using Corollary~\ref{AB}:
\begin{corollary}
	Let $f$ belongs to $\mathcal{C}[1-2\alpha,-1]$, where $0\leq \alpha<1$. Then the logarithmic coefficients of $f$ satisfies \eqref{logBohrineq} for
	\begin{equation*}
	|z|=r \leq 1-\exp(1/(\alpha-1)), \;\text{when}\; \alpha\neq 1/2
	\end{equation*}
	and 
	\begin{equation*}
	|z|=r \leq 1-1/\exp(2), \;\text{when}\; \alpha= 1/2.
	\end{equation*}
\end{corollary} 

\begin{corollary}\label{Janwoski-Disk}
	Let $f$ belongs to $\mathcal{C}[D,0]$. Then the logarithmic coefficients of $f$ satisfies \eqref{logBohrineq} for
	\begin{equation*}
	|z|=r \leq 1-\exp(-\tfrac{2}{D}),
	\end{equation*}
\end{corollary}
\begin{proof}
	From the proof of Corollary~\ref{AB}, we obtain that 
	$$\psi(z)=Dz e^{Dz}/(e^{Dz}-1)=1+\frac{D}{2}z+\frac{D^2}{2}z^2- \cdots,$$ when $\phi(z)=1+Dz.$ Now with a little computation, we find that the function $l(z)=ze^z/(e^z-1)$ is convex univalent in $\mathbb{D}$. Therefore, the function $\psi(z)=l(Dz)$ is also convex in $\mathbb{D}$ for each fixed $0<D\leq1$.   \qed
\end{proof}

\begin{theorem}\label{hallen}
	Let $\phi$ be convex in $\mathbb{D}$ with $\Re \phi(z)>0$, and suppose $f\in \mathcal{A}$ satisfies the differential subordination
	\begin{equation}\label{halb}
	\frac{zf'(z)}{f(z)}+z\left(\frac{zf'(z)}{f(z)}\right)' \prec \phi(z)=1+B_1z+B_2 z^2+ \cdots.
	\end{equation}
	Then the logarithmic coefficients of $f$ satisfies \eqref{logBohrineq} for
	\begin{equation*}
	|z|=r \leq  1-\frac{1}{\exp\big(\frac{2}{|B_1|}\big)}.
	\end{equation*}
	The result is sharp for the case when the function $f$ defined as $zf'(z)/f(z)=\psi(z)$ is an extremal for the logarithmic coefficient's bounds.
	
\end{theorem}
\begin{proof}
	Let $p(z)=zf'(z)/f(z)$. Then the subordination \eqref{halb} can equivalently be written as:
	$p(z)+zp'(z) \prec \phi(z)$.	
	A simple calculation show that the analytic function $$\psi(z):=({1}/{z})\int_{0}^{z}{\phi(t)}dt= 1+\frac{B_1}{2}z+\frac{B_2}{3}z^2+\cdots$$
	satisfies
	\begin{equation*}
	\psi(z)+z\psi'(z)=\phi(z).
	\end{equation*}
	Now from the Hallenbeck and Ruscheweyh result \cite[Theorem~ 3.1b, p.~71]{subbook}, we have $p\prec \psi$, where $\psi$ is the best dominant and also convex. Further, since $\Re\phi(z)>0$, using the integral operator \cite[Theorem~ 4.2a, p.~202]{subbook} preserving functions with positive real part, we see that $\psi$ is a Carathe\'{o}dory function. Thus we have
	$$\frac{zf'(z)}{f(z)} \prec \psi(z).$$
	Now applying Theorem~\ref{logbohrstarlike}, the result follows.   \qed
\end{proof}		

\begin{corollary}
	Suppose $f\in \mathcal{A}$ satisfies the differential subordination
	\begin{equation*}
	\frac{zf'(z)}{f(z)}+z\left(\frac{zf'(z)}{f(z)}\right)' \prec \frac{1+z}{1-z}.
	\end{equation*}
	Then the logarithmic coefficients of $f$ satisfies \eqref{logBohrineq} for
	$|z|=r \leq (e-1)/e\approx 0.632121,$
\end{corollary}

\begin{corollary}
	Suppose $f\in \mathcal{A}$ satisfies the differential subordination
	\begin{equation*}
	\frac{zf'(z)}{f(z)}+z\left(\frac{zf'(z)}{f(z)}\right)' \prec e^z.
	\end{equation*}
	Then the logarithmic coefficients of $f$ satisfies \eqref{logBohrineq} for
	$|z|=r \leq (e^2-1)/e^2 \approx 0.864665,$
\end{corollary}

\begin{theorem}
	Let $\phi$ be convex in $\mathbb{D}$ with $\Re \phi(z)>0$, and suppose $f\in \mathcal{A}$ satisfies the differential subordination
	\begin{equation}\label{p2p}
	\frac{zf'(z)}{f(z)}\left(\frac{zf'(z)}{f(z)}+2z\left(\frac{zf'(z)}{f(z)}\right)'\right) \prec \phi(z)=1+B_1z+B_2 z^2+ \cdots, B_1\neq 0.
	\end{equation}
	Then the logarithmic coefficients of $f$ satisfies \eqref{logBohrineq} for
	\begin{equation*}
	|z|=r \leq 1-\frac{1}{\exp(\frac{4}{|B_1|})}.
	\end{equation*}
	The result is sharp for the case when the function $f$ defined as $zf'(z)/f(z)=\sqrt{\psi(z)}$ is an extremal for the logarithmic coefficient's bounds.
\end{theorem}
\begin{proof}
	Let $p(z)=zf'(z)/f(z)$. Then the subordination \eqref{p2p} can be equivalently written as:
	\begin{equation*}
	p^2(z)+2 zp(z)p'(z) \prec \phi(z),
	\end{equation*}
	which using the change of variable $P(z)=p^2(z)$ becomes
	\begin{equation*}
	P(z)+zP'(z) \prec \phi(z).
	\end{equation*}
	Now proceeding as in Theorem~\ref{hallen}, we see that $p(z) \prec \sqrt{\psi(z)}$ and $\sqrt{\psi(z)}$ is the best dominant. Further, since $\Re\phi(z)>0$, using \cite[Theorem~ 4.2a, p.~202]{subbook}, we see that $\psi$ is a Carath\'{e}odory function. Thus we have
	\begin{equation*}
	\frac{zf'(z)}{f(z)} \prec {\sqrt{\psi(z)}}=1+\frac{B_1}{4}z+\left(\frac{B_2}{6}-\frac{B^2_1}{32} \right)z^2+\cdots,
	\end{equation*}
	where $\psi(z)=\frac{1}{z}\int_{0}^{z}\phi(t)dt$. Now if we let $q(z)=1+z\psi''(z)/\psi'(z)$. Then
	\begin{equation*}
	q(0)=1\quad  \text{and} \quad \psi'(z)(q(z)+1)=\phi'(z),
	\end{equation*}
	which further by logarithmic differentiation gives
	\begin{equation}\label{psiconvexlog}
	\Psi(r,s;z):=q(z)+\frac{zq'(z)}{q(z)+1} = 1+ \frac{z\phi''(z)}{\phi'(z)}.
	\end{equation} 
	Since the function $\Psi$ satisfy the admissible conditions \cite[2.3-11, p.~35]{subbook}, using \cite[Theorem~ 2.3i (i), p.~35]{subbook} in \eqref{psiconvexlog}, we see  that $\psi$ is convex. 
	Let 
	$$h(z)=(\psi(z)-1)/\psi'(0)\in \mathcal{A}.$$
	Now it is easy to conclude that $h\in \mathcal{S}^* $, hence univalent. Finally, note that 
	$$|\arg{\sqrt{\psi(z)}}|=\frac{1}{2}|\arg{\psi(z)}| \leq \frac{\pi}{4},$$
	which implies $\Re\sqrt{\psi(z)}>0$. Thus we see that $\sqrt{\psi(z)}$ is a  Carath\'{e}odory univalent convex function. Now applying Theorem~\ref{logbohrstarlike}, the result follows.   \qed
\end{proof}		

\begin{corollary}
	Suppose $f\in \mathcal{A}$ satisfies the differential subordination
	\begin{equation*}
	\frac{zf'(z)}{f(z)}\left(\frac{zf'(z)}{f(z)}+2z\left(\frac{zf'(z)}{f(z)}\right)'\right) \prec \frac{1+(2\alpha -1)z}{1+z},\; \alpha\in [0,1).
	\end{equation*}
	Then the logarithmic coefficients of $f$ satisfies \eqref{logBohrineq} for $|z|=r \leq 1-\exp\left(\frac{2}{\alpha-1} \right).$
\end{corollary}

\begin{corollary}
	Suppose $f\in \mathcal{A}$ satisfies the differential subordination
	\begin{equation*}
	\frac{zf'(z)}{f(z)}\left(\frac{zf'(z)}{f(z)}+2z\left(\frac{zf'(z)}{f(z)}\right)'\right) \prec 1+\alpha z, \quad \alpha\in (0,1].
	\end{equation*}
	Then the logarithmic coefficients of $f$ satisfies \eqref{logBohrineq} for $|z|=r \leq 1-1/\exp(4/\alpha).$
\end{corollary}


%
\section*{Conflict of interest}

The authors declare that they have no conflict of interest.



\end{document}